\documentclass[a4paper, twoside, 10pt]{article}
\usepackage{eurosym}
\usepackage{amsmath,amsthm,amsfonts,latexsym,amscd,amssymb, enumerate}
\usepackage{hyperref}
\numberwithin{equation}{section}
\input xypic
\newtheorem{theorem}{Theorem}[section]
\newtheorem{lemma}{Lemma}[section]

\newtheorem{definition}{Definition}[section]
\newtheorem{example}{Example}[section]
\theoremstyle{remark}

\date{}
	\title{\textbf{Clairaut slant Riemannian maps to K\"ahler manifolds}}
\author{Jyoti Yadav, Gauree Shanker*\thanks{corresponding author, Email: gauree.shanker@cup.edu.in} and Murat Polat }

\begin{document}
\maketitle
\begin{abstract}
	The aim of this article is to describe the idea of  Clairaut slant Riemannian maps from  Riemannian manifolds to  K\"ahler manifolds. First, for the slant Riemannian map, we obtain the necessary and sufficient
	conditions for a curve to be a geodesic on the base manifold. Further, we find the necessary and sufficient conditions for the  slant Riemannian map to be a Clairaut slant Riemannian map; for Clairaut slant Riemannian map to be totally geodesic; for the base manifold to be a locally product manifold. Further, we obtain the necessary and sufficient condition for the integrability of range of derivative map. Also, we investigate the harmonicity of Clairaut slant Riemannian map. Finally, we get two inequalities in terms of second fundamental form of a Clairaut slant Riemannian map and check the equality case.
\end{abstract}

\noindent\textbf {M. S. C. 2020:} 53B20, 53B35.\\
\textbf{Keywords:}  K\"ahler manifold, Clairaut submersion, Riemannian map, slant Riemannian map, Clairaut Riemannian map.

\section{Introduction}
To understand the relation between the geometric structures of two Riemannian manifolds, smooth maps are helpful. Therefore there is a need to define more maps for comparing the geometric properties between two Riemannian manifolds. Riemannian submersions and isometric immersions  are the two such types of basic
maps. A smooth map $\pi$ between Riemannian manifolds $(P, g_P)$ and $(Q, g_Q)$ is said to be an
isometric immersion if the differential map $\pi_*$ is one-one and satisfies the condition
$g_Q(\pi_*X, \pi_*Y) = g_P (X, Y)$ for $X, Y \in \Gamma(TP).$ 
O'Neill \cite{S10} and Gray \cite{S35} addressed Riemannian submersion
and O'Neill derived the fundamental equations for Riemannian submersion, which are helpful to study
the geometry of Riemannian manifolds. A smooth map $\pi$ between Riemannian
manifolds $P$ and $Q$ is called a Riemannian submersion if $\pi_*$ is surjective  and it preserves the length of the horizontal vector field. In \cite{S28, S27}, the geometry of Riemannian submersion was investigated.\\ 
In \cite{S3}, Fischer proposed the idea of Riemannian maps between Riemannian manifolds which are the generalizations of Riemannian submersions and isometric immersions.
A notable characteristic of Riemannian maps is that  Riemannian maps satisfy the generalized
eikonal equation  $||\pi_*||^2= rank\pi$  which connects geometric optics and physical
optics. Fischer also demonstrated how Riemannian maps can be used to build various quantum
models of nature. The geometry of Riemannian maps were investigated in \cite{S29, S23, S30}. There are several kinds of submanifolds which depend on how a tangent bundle of a 
submanifold responds to the influence of the complex structure J$'$ of the ambient
manifold, namely, K\"ahler submanifolds, CR-submanifolds, totally real submanifolds, generic submanifolds,
slant submanifolds, pointwise slant submanifolds, semi-slant submanifolds, and hemi-slant
submanifolds. Chen \cite{S15} introduced the idea of slant submanifolds of an almost Hermitian manifold, and according to this idea, slant submanifolds include totally real and holomorphic submanifolds.
In addition, \c{S}ahin \cite{S16} provided the concept of slant Riemannian map which is  a generalization of slant immersions ( totally real immersions
and holomorphic immersions), invariant Riemannian maps and anti-invariant Riemannian
maps.\\ The idea of Clairaut's relation comes from elementary differential	 geometry.
According to Clairaut theorem, let $\rho$ be the
distance of the surface from the axis of rotation and let $\theta$ be the angle between the meridian and velocity
vector of the geodesic on the surface then $\rho sin\theta$ is constant. Bishop \cite{S5}  generalized this idea on submersion theory and introduced Clairaut submersion. A submersion $\pi: P \rightarrow Q$ is said to be a Clairaut submersion if there is a function $\rho : P \rightarrow R^+$ such that for every geodesic, making an angle $\theta$ with the horizontal subspace then $\rho sin\theta$ is constant. Further, Clairaut submersion has been studiedin \cite{S31} and in many other spaces viz. Lorentzian spaces, timelike and spacelike spaces \cite{S14, S11, S13, S34}. Since then, submersions have been defined on different aspects. A Clairaut submersion is a helpful tool for establishing decomposition theorems on Riemannian manifolds.  Many submersions are based on the behavior of tangent bundle of ambient space and submanifolds. Watson \cite{S6} introduced the concept of almost Hermitian submersion. \c{S}ahin \cite{S8} introduced holomorphic Riemannian maps between almost Hermitian manifolds, which is a generalization of holomorphic submanifolds and holomorphic submersions. Additionally, the idea of a Riemannian map has been examined from different points of view viz. anti-invariant \cite{S4}, semi-invariant \cite{S7} , slant Riemannian maps \cite{S16} from a Riemannian manifold to a K\"ahler manifold. Further, conformal anti-invariant, semi-invariant Riemannian map from Riemannian manifold to K\"ahler manifold were introduced in \cite{S26, S32}. The Clairaut Riemannian maps were introduced in \cite{S2}, \cite{S25} and \cite{Kiran_Thesis}. Recently, Meena et. al have introduced Clairaut invariant \cite{S19}, anti-invariant \cite{S20}, and semi-invariant \cite{S21} Riemannian maps between Riemannian manifolds and K\"ahler manifolds.\\ 
In this article, we discuss about Clairaut slant Riemannian maps.
The paper is structured as follows: In section 2, we discuss several fundamental terms, definitions, and informations required for this paper. In section 3, we define Clairaut slant
Riemannian map from a Riemannian manifold to a K\"ahler manifold.
Further, we obtain a necessary and sufficient condition for a slant
Riemannian map to be Clairaut. Additionally, we find a necessary and sufficient condition for the Clairaut slant Riemannian map to be totally geodesic. Next we find harmonicity of Clairaut slant Riemannian map. Along with this, we obtain inequalities of Clairaut slant Riemannian map and check the inequality case.
Finally, we provide a non-trivial example for existence of such Clairaut slant Riemannian
map.

\section{Preliminaries}
Let $\pi$ be a smooth map between two Riemannian manifolds $(P, g_P)$ and $(Q, g_Q)$ of dimension $p, q$ respectively, such that $rank \pi \leq\min\{p,q\}$. Let $\mathcal{V}_r = ker\pi_{*r}$ at $r\in P,$ stands for vertical distribution or kernel space of $\pi_*$ and $\mathcal{H}_r = (ker\pi_{*r} )^\perp$ in $T_rP$ is the orthogonal complementary space of $\mathcal{V}_r$. 
Then the tangent space $T_rP$ at $r\in P$  has the decomposition,\\
$$ T_rP = (ker\pi_{*r}) \oplus (ker\pi_{*r} )^\perp = \mathcal{V}_r \oplus \mathcal{H}_r.$$
Let the range of $\pi_*$ be denoted by $range\pi_{*r}$ at $r\in P$, and $(range\pi_{*r})^\perp$ be the orthogonal complementary space of $range\pi_{*r}$ in the tangent space $T_{\pi(r)}Q$ of $Q$ at $\pi(r)\in Q$. Since $rank\pi \leq\min\{p,q\}$, this gives $(range\pi_{*r})^\perp\neq {0}$. Thus, the tangent space $T_{\pi(r)}Q$ of $Q$ at $\pi(r)\in Q$ has the following decomposition:
\begin{equation*}
	T_{\pi(r)}Q = (range\pi_{*r})\oplus (range\pi_{*r})^\perp.
\end{equation*}
Now, a smooth map $\pi : (P, g_P)\rightarrow (Q, g_Q)$ is said to be a Riemannian map at $r_1\in P$ if the horizontal restriction
$\pi^h_{r_1}:(ker\pi_{*r_{1}})^\perp\rightarrow (range\pi_{*r_{1}})$ is a linear isometry between the inner product spaces $((ker\pi_{*r_{1}})^\perp, g_Q{(r_{1})}| (ker\pi_{*r_{1}})^\perp )$ and $((range\pi_{*r_1}, g_Q (r_{2}) |(range\pi_{*r_1})), r_2 = \pi(r_{1})$. In another words, $\pi_*$
satisfies the equation\\
\begin{equation}{\label{N1}}
	g_Q (\pi_{*}X, \pi_{*}Y) = g_P (X, Y), \forall X, Y \in\Gamma(ker\pi_*)^\perp.
\end{equation}
It is observed  that Riemannian submersions  and isometric immersions are particular case of  Riemannian maps with $(range\pi_{*r})^\perp = 0$  and $ker\pi_{*r} = 0$, respectively.\\
For any vector field $X$ on $P$ and any section $V$ of $(range\pi_{*})^\perp$, we define $\nabla^{\pi \perp}_ {X} V$ as the orthogonal projection of $\nabla^{Q}_{X} V$ on $(range\pi_{*})^\perp$.\\\\
From now,  we denote by $\nabla^Q$ Levi-Civita connection for
$(Q, g_Q)$ and $\nabla^{Q_\pi}$ pullback connection along $\pi$.
Next, suppose that $\pi$ is a Riemannian map and define $S_V$ as \cite{S23}
\begin{equation}\label{S_V}
	\nabla^{Q}_{\pi_{*}X}V = -S_V{\pi_{*}X} + \nabla^{\pi \perp}_ {X} V,
\end{equation}
where $S_V{\pi_{*}X}$ is the tangential component and  $\nabla^{\pi \perp}_ {X} V$ the orthogonal component  of $\nabla^{Q}_{\pi_{*}X}V$.
It can be easily seen that $\nabla^{Q}_{\pi_{*}X}V$ is obtained from the pullback connection of $\nabla^{Q}$. Thus, at $r\in P$, we have 
$\nabla^{Q}_{\pi_{*}X}V(r) \in T_{\pi(r)}Q, S_V{\pi_{*}X}\in\pi_{*r}(T_{r}P)$ and $\nabla^{\pi \perp}_ {X} V\in (\pi_{*r}(T_{r}P))^\perp$. It follows that $S_V{\pi_{*}X}$ is bilinear in $V$ and $\pi_{*}X$ and $S_V{\pi_{*}X}$  at $r$ depends only on $V_{r}$ and $\pi_{*r}X_{r}$.\\  By direct computations, we obtain 
\begin{equation}\label{Sv}
	g_Q (S_V{\pi_{*}X, {\pi_{*}Y}}) = g_Q \big(V, (\nabla \pi_{*})(X, Y)\big),\\ 
\end{equation} 
for all X, Y $\in \Gamma(ker\pi_{*})^\perp~and ~V\in\Gamma(range\pi_{*})^\perp.$
Since $(\nabla \pi_{*})$ is symmetric, it follows
that $S_V$ is a symmetric linear transformation of $range \pi_{*}$.  \\\
Let $\pi : (P, g_P) \rightarrow (Q, g_Q)$ be a smooth map between manifolds $(P, g_P)$ and $(Q, g_Q)$. The second fundamental form of $\pi$ is the map \cite{S22}
\begin{equation*}
	\nabla \pi_{*} : \Gamma(TP)\times \Gamma(TP) \rightarrow \Gamma_{\pi}(TQ)\\
\end{equation*}
defined by
\begin{equation}\label{SFF}
	(\nabla \pi_{*}) (X, Y) = \nabla^{Q_{\pi}}_{X}{\pi_{*}}Y - \pi_{*}(\nabla^P_{X} Y),
\end{equation}
where  $\nabla^P$ is a linear connection on $P$.\\
For any $X, Y \in \Gamma(ker\pi_{*})^\perp,$ \c{S}ahin\cite{S4} showed that the second fundamental form $(\nabla \pi_{*}) (X, Y)$   of a Riemannian
map has no components in $range\pi_{*}$. It means 
\begin{equation}
	(\nabla \pi_{*}) (X, Y)\in \Gamma(range \pi_{*})^\perp.\\\\
\end{equation}
Trace of second fundamental form $\pi$ is called tension field \cite{S24}. It is denoted by $\tau(\pi)$ and defined as $\tau(\pi) = trace(\nabla \pi_{*}) = \sum_{i=1}^{m}(\nabla \pi_{*})(e_{i}, e_{i}).$
A map $\pi$ is called a harmonic map \cite{S24} if it has a vanishing tension field, i.e., $\tau(\pi) = 0.$\\
The adjoint map *$\pi_{*r}$ at $r\in P $ of the map $\pi$ is defined by
\begin{equation}
	g_Q(\pi_{*p}(X), W) = g_P(X,*\pi_{*r}(W))\\ 
\end{equation}
for $X \in T_rP$ and $W\in T_{\pi(r)}Q$, where $\pi_{*r}$ is the derivative of $\pi$ at $r\in P$.\\\\
\begin{lemma}\label{Lemma}\cite{S7} Let $\pi : (P, g_P ) \rightarrow (Q, g_Q )$ be a Riemannian map between Riemannian manifolds. Then $\pi$ is umbilical Riemannian map if and only if $(\nabla \pi_*)(X, Y ) = g_P (X, Y )H$ for all $X, Y \in \Gamma (ker\pi_*)^\perp$ and $H$ is, nowhere zero, mean curvature vector field on $(range\pi_*)^\perp$. 
\end{lemma}
Let $(P,g_P)$ be an almost Hermitian manifold \cite{S1}, then $P$ admits a
tensor $J'$ of type (1,1) such that $J'^2 = -I$ and
\begin{equation}\label{COHM}
	g_P(J'X, J'Y) = g_P(X, Y)
\end{equation}
for all $X, Y \in \Gamma(TP).$ An almost Hermitian manifold $P$ is called K\"ahler manifold if
\begin{equation}\label{COKM}
	(\nabla_{X}J')Y = 0,\\
\end{equation}
for all $X, Y \in \Gamma(TP)$, where $\nabla$ is Levi-Civita connection on $P$.
\begin{definition}
	\cite{S16}
	Let $\pi$ be a Riemannian map from a Riemannian manifold $(P, g_P)$
	to an almost Hermitian manifold $(Q, g_Q, J')$. If for any non-zero vector $X \in \Gamma(ker\pi_*)^\perp,$
	the angle $\theta (X)$ between $J'\pi_*(X)$ and the space $range\pi_*$ is a constant, i.e., it is independent of the choice of the point $r \in P$ and choice of the tangent vector $\pi_*(X)$ in
	$range\pi_*$, then we say that $\pi$ is a slant Riemannian map. In this case, the angle $\theta$ is called the slant angle of the slant Riemannian map.
\end{definition}
Let $\pi$ be a Riemannian map from a Riemannian manifold $(P, g_P)$ to an almost Hermitian manifold $(Q, g_Q, J' )$. Then for $\pi_*Y \in \Gamma(range\pi_*),
Y \in \Gamma(ker \pi_*)^\perp$,
we have
\begin{equation}\label{DJF}
	J'\pi_*Y = \alpha \pi_*Y + \delta \pi_*Y,  	
\end{equation}
where $\alpha \pi_*Y \in \Gamma(range\pi_*)$ and $\delta \pi_*Y \in \Gamma (range\pi_*)^\perp$.
Also for $U \in \Gamma(range\pi_*)^\perp,$
we have
\begin{equation}\label{DJV}
	J'U = BU + CU,
\end{equation}
where $BU \in \Gamma (range\pi_*)$ and $CU \in \Gamma(range\pi_*)^\perp.$\\
The idea of Clairaut Riemannian map is based on geodesic of surface of revolution. \c{S}ahin \cite{S2} defined Clairaut Riemannian map  by using geodesics on total manifolds.
A Riemannian map $\pi : P \rightarrow Q$ between Riemannian manifolds
$(P, g_P )$ and $(Q, g_Q)$ is called a Clairaut Riemannian map if there is a function $s : P \rightarrow R^+$ such that for every geodesic, making angles $\theta$ with the horizontal subspaces, $ssin\theta$ is constant.\\
Similarly, Clairaut Riemannian map has been defined by using geodesic on base manifold as follows.
\begin{definition}\label{CRMps}\cite{S25}
	A Riemannian map $\pi : (P, g_P ) \rightarrow (Q, g_Q )$ between Riemannian manifolds is called Clairaut Riemannian map if there is a function $s: Q \rightarrow R^+$ such that for every geodesic $\eta$ on $Q$ , the function 
	$(s o \eta) sin \omega(t)$  is constant, where $\pi_*X \in \Gamma(range\pi_*)$ and $U \in \Gamma(range \pi_*)^\perp$ are the vertical and horizontal components of $\dot\eta(t)$, and $\omega(t)$ is the angle between $\dot\eta(t)$ and $U$ for all t.
\end{definition}
\begin{theorem}\cite{S25}\label{NSC}
	Let $\pi : (P, g_P ) \rightarrow (Q, g_Q)$ be a Riemannian map between Riemannian manifolds such that $(range\pi_*)^\perp$ is totally geodesic and $range\pi_*$ is connected, and let $\beta, \eta = \pi o \beta$ be geodesics on $P$ and $Q,$ respectively.Then $\pi$ is Clairaut Riemannian map with $s = e^f$ if and only if any one of the following conditions holds: 
\end{theorem}
\begin{itemize}
	\item[(i)] $S_V \pi_*X = -V(f) \pi_*X$, where $\pi_*X \in\Gamma (range\pi_*)$ and 
	$V \in\Gamma(range\pi_*)^\perp$ are components of  $\dot{\eta}(t)$. 
	\item [(ii)] $\pi$ is umbilical map, and has $H = -\nabla^Q f$ , where $g$ is a smooth function on $Q$ and $H$ is the mean curvature vector field of $range\pi_*$.
\end{itemize} 
\begin{theorem}\cite{S16}\label{NSCfor Slant}
	Let $\pi$ be a Riemannian map from a Riemannian manifold $(P, g_P)$
	to an almost Hermitian manifold $(Q, g_Q, J')$. Then $\pi$ is a slant Riemannian map if and only if there exists a constant $\lambda \in [-1, 0]$ such that
	$\alpha^2\pi_*(X) = -\lambda \pi_*(X)$
	for $X \in \Gamma((ker \pi_*)^\perp).$ If $\pi$ is a slant Riemannian map, then $\lambda = - cos^2\theta$.
\end{theorem}

\section{Clairaut slant Riemannian maps to K\"ahler manifolds}

In this section, we introduce the notion of Clairaut slant Riemannian map from a Riemannian manifold to a K\"ahler manifold. We investigate some characteristics of this map. 
\begin{definition}
	A slant Riemannian map from a Riemannian  manifold $(P, g_P)$ to a K\"ahler manifold $(Q, g_Q)$ is called a Clairaut slant Riemannian map if it satisfies the definition \ref{CRMps}.
\end{definition}
\noindent Next, onward, we will consider the map $\pi$ in which $(range\pi_*)^\perp$ is totally geodesic.
\begin{theorem}
	Let $\pi$ be a slant Riemannian map from a Riemannian manifold $(P, g_P)$ to a K\"ahler manifold $(Q, g_Q, J')$. If $\beta$ is a geodesic  on $(P, g_P)$, then the curve $\eta = \pi o \beta$ is geodesic on $Q$ if and only if 
	\begin{equation}\label{GCCS1}
		\begin{split}
			&cos^2\theta \pi_*(\nabla^P_X X) + S_{\delta(\alpha \pi_{*}X)}\pi_*X - \pi_*(\nabla^P_X*\pi_{*}B(\delta \pi_{*}X))+ 	S_{C(\delta \pi_*X)} \pi_*X\\& + cos^2\theta\nabla^Q_U \pi_*X - \nabla^\pi_U B(\delta \pi_{*}X) - B(\nabla \pi_*)(X, *\pi_{*}BU) - \alpha(\nabla^P_X *\pi_{*}BU)\\& + \alpha(S_{CU}\pi_*X) - B(\nabla^{\pi\perp}_X CU) -\alpha(\nabla^\pi_U BU) - B(\nabla^{\pi\perp}_U CU) = 0
		\end{split}
	\end{equation} 
	and 
	\begin{equation}\label{GCCS2}
		\begin{split}
			&cos^2\theta(\nabla \pi_*)(X, X) - \nabla^{\pi\perp}_X \delta(\alpha \pi_* X) - (\nabla \pi_*)(X, *\pi_{*}B(\delta \pi_*X))\\& - \nabla^{\pi\perp}_X C(\delta \pi_*X) - \nabla^{\pi\perp}_U \delta(\pi_{*}X) - \nabla^{\pi\perp}_U C(\delta \pi_*X) - C(\nabla \pi_*)(X, *\pi_{*}BU)\\& + \delta(\nabla^P_X *\pi_{*}BU) - \delta\nabla^Q_U BU - C\nabla^{\pi\perp}_U CU - C\nabla^{\pi\perp}_X CU + \delta(S_{CU}\pi_*X) = 0, 
		\end{split}
	\end{equation}
	where $\pi_*X, U$ are the vertical and horizontal part of $\dot{\eta}$ repectively,and $\nabla^Q$ is the Levi-Civita connection on $Q$ and $\nabla^{\pi\perp}$ is a linear connection on $(range\pi_*)^\perp$.
	
\end{theorem}
\begin{proof}
	Let $\beta$ is a geodesic on $P$ and let $\eta = \pi o \beta$ is a regular curve on $Q$. Suppose $\pi_*X$, $U$ are the vertical and horizontal components respectively, of $\dot{\eta}(t)$.  Since $Q$ is K\"ahler manifold, we can write 
	$$\nabla^Q_{\dot{\eta}}{\dot{\eta}} = -J'\nabla^Q_{\dot{\eta}}{J'\dot {\eta}}$$
	which implies 
	\begin{eqnarray}
		\nabla^Q_{\dot{\eta}}{\dot{\eta}} &=& -J'\nabla^Q_{\pi_*X + U }{J'(\pi_*X +U)},
		\nonumber\\&=& -J'\nabla^Q_{\pi_*X}J'\pi_*X - J'\nabla^Q_{U}J'\pi_*X - J'\nabla^Q_{\pi_*X}J'U - J'\nabla^Q_{U}J'U.\nonumber
	\end{eqnarray}
	Making use of \eqref{COKM}, \eqref{DJF} and \eqref{DJV} in above equation, we get
	
	\begin{equation}\label{J1}
		\begin{split}
			\nabla^Q_{\dot{\eta}}{\dot{\eta}}& = -\nabla^Q_{\pi_*X}\alpha^2\pi_*X - \nabla^Q_{\pi_*X}\delta(\alpha \pi_*X) - \nabla^Q_{\pi_*X}B(\delta \pi_*X) - \nabla^Q_{\pi_*X}C(\delta \pi_*X) \\& - \nabla^Q_U \alpha^2\pi_*X - \nabla^Q_U \delta(\alpha \pi_*X) - \nabla^Q_U B(\delta \pi_*X) - \nabla^Q_U C(\delta  \pi_*X) \\&- J'\nabla^Q_{\pi_*X}BU - J'\nabla^Q_{\pi_*X}CU - J'\nabla^Q_{U}BU - J'\nabla^Q_{U}CU.	
		\end{split}
	\end{equation}
	From \eqref{Sv} and \eqref{SFF}, we get
	$$\nabla^Q_{\pi_*X}\pi_*X = (\nabla \pi_*)(X, X) +  \pi_*(\nabla^P_X X),$$ 
	$$\nabla^Q_{\pi_*X}\delta(\alpha\pi_*X) = -S_{\delta(\alpha \pi_{*}X)}\pi_*X + \nabla^{\pi\perp}_X \delta(\alpha \pi_*X),$$
	$$\nabla^Q_{\pi_{*}X}B(\delta \pi_*X) = (\nabla \pi_*)(X, *_{\pi_*}B(\delta \pi_*X)) + \pi_*(\nabla_X *_{\pi_*}B(\delta \pi_*X)),$$
	$$\nabla^Q_{\pi_*X}C(\delta \pi_*X) = -S_{C(\delta \pi_{*}X)}\pi_*X + \nabla^{\pi\perp}_X C(\delta \pi_*X),$$
	$$\nabla^Q_{\pi_*X}CU = -S_{CU}\pi_{*}X + \nabla^{\pi\perp}_X CU,$$
	$$\nabla^Q_{\pi_*X}BU = (\nabla \pi_*)(X, *_{\pi_*}BU) + \pi_*(\nabla^P_X *_{\pi_{*}}BU).$$
	Since $(range\pi_*)^\perp$ is totally geodesic,
	
	$$\nabla^Q_U\delta(\alpha \pi_*X) = \nabla^{\pi\perp}_U\delta(\alpha \pi_*X),$$
	$$\nabla^Q_U C(\delta \pi_*X) = \nabla^{Q\perp}_UC(\delta \pi_*X),$$
	$$\nabla^Q_U CU = \nabla^{\pi\perp}_U CU.$$
	By metric compatibility, we have\\
	$$\nabla^Q_UB(\delta \pi_*X)\in\Gamma(range \pi_*),$$
	$$\nabla^Q_U BU\in\Gamma(range\pi_*),$$
	$$\nabla^Q_U \pi_*X \in\Gamma(range\pi_*).$$
	Making use of above terms in \eqref{J1}, we obtain
	
	\begin{equation}
		\begin{split}
			\nabla^Q_{\dot{\eta}}{\dot{\eta}}& = cos^2\theta(\nabla \pi_*)(X, X) + cos^2\theta \pi_*(\nabla^P_X X) + S_{\delta(\alpha \pi_* X)}\pi_*X \\&- \nabla^{\pi\perp}_X \delta(\alpha \pi_*X) + cos^2\theta\nabla^Q_U \pi_*X -\nabla^{\pi\perp}_U \delta(\alpha \pi_* X) \\&- (\nabla \pi_*)(X, *_{\pi_*}B(\delta \pi_*X)) - \pi_*(\nabla^P_X *_{\pi_*}B(\delta \pi_* X)) + S_{C(\delta\pi_*X)} \pi_*X \\&- \nabla^{\pi\perp}_X C(\delta \pi_*X) - \nabla^Q_U B(\delta \pi_*X) - \nabla^{\pi\perp}_U C(\delta \pi_*X) - B(\nabla \pi_*)(X, *_{\pi_*}BU) \\&- C(\nabla \pi_*)(X, *_{\pi_*}BU) -\alpha \pi_*(\nabla^P_X*_{\pi_*}BU)-\delta \pi_*(\nabla^P_X*_{\pi_*}BU) \\&+ \alpha(S_{CU} \pi_*X) + \delta(S_{CU} \pi_*X)- B(\nabla^{\pi\perp}_X CU) - C(\nabla^{\pi\perp}_X CU) \\&-\alpha(\nabla^Q_UBU) - \delta(\nabla^Q_U BU)- B(\nabla^{\pi\perp}_U CU) - C(\nabla^{\pi\perp}_U CU).
		\end{split}
	\end{equation}
	Since $\eta$ is geodesic on $Q$ 
	if and only if $\nabla^Q_{\dot{\eta}} \dot{\eta} = 0,$\\ this implies
	\begin{equation}
		\begin{split}
			 &cos^2\theta(\nabla \pi_*)(X, X) + cos^2\theta \pi_*(\nabla^P_X X) + S_{\delta(\alpha \pi_* X)}\pi_*X - \nabla^{\pi\perp}_X \delta(\alpha \pi_*X) \\&+ cos^2\theta\nabla^Q_U \pi_*X -\nabla^{\pi\perp}_U \delta(\alpha \pi_* X) - (\nabla \pi_*)(X, *_{\pi_*}B(\delta \pi_*X)) \\&- \pi_*(\nabla^P_X *_{\pi_*}B(\delta \pi_* X)) + S_{C(\delta\pi_*X)} \pi_*X - \nabla^{\pi\perp}_X C(\delta \pi_*X) \\&- \nabla^Q_U B(\delta \pi_*X) - \nabla^{\pi\perp}_U C(\delta \pi_*X) - B(\nabla \pi_*)(X, *_{\pi_*}BU) \\&- C(\nabla \pi_*)(X, *_{\pi_*}BU)-\alpha \pi_*(\nabla^P_X*_{\pi_*}BU)-\delta \pi_*(\nabla^P_X*_{\pi_*}BU) \\& + \alpha(S_{CU} \pi_*X) + \delta(S_{CU} \pi_*X)- B(\nabla^{\pi\perp}_X CU) - C(\nabla^{\pi\perp}_X CU)\\& -\alpha(\nabla^Q_UBU) - \delta(\nabla^Q_U BU)- B(\nabla^{\pi\perp}_U CU) - C(\nabla^{\pi\perp}_U CU) =0.
		\end{split}
	\end{equation}
	Taking vertical and horizontal part, we get \eqref{GCCS1} and \eqref{GCCS2}.
\end{proof}
\begin{theorem}
	Let $\pi$ be a slant Riemannian map with connected fibers from a Riemannian manifold $(P, g_P)$ to a K\"ahler manifold $(Q, g_Q, J')$. Then $\pi$ is a Clairaut slant Riemannian map with $s=e^f$ if and only if 
	\begin{equation}
		\begin{split}
			&B(\nabla \pi_*)(X, *_{\pi_*}BU) + \alpha (\nabla^P_X *_{\pi_*}BU) - \alpha(S_{CU}\pi_*X) + B(\nabla^{\pi\perp}_X CU)\\& + \alpha (\nabla^Q_U BU) + B(\nabla^{\pi\perp}_U CU) =  \pi_*{Y} U(f),
		\end{split}
	\end{equation}
	where $\pi_*X, U$ are the vertical and horizontal part of tangent vector field $\dot{\eta}$ on $Q$.
\end{theorem}
\begin{proof}
	Let $\eta$ be a geodesic on $Q$ with constant speed $k$ i.e., k = $||\dot{\eta}||^2$.\\
	Let vertical and horizontal part of $\dot{\eta}$ are $\pi_*X, U,$ respectively. Then, we get
	
	\begin{equation}\label{||(F_*x)||^2}
		g_Q(\pi_*X, \pi_*X) = ksin^2\omega(t)
	\end{equation}
	and
	\begin{equation}
		g_Q(U, U) = k cos^2\omega(t),
	\end{equation}
	where $\omega(t)$ is the angle between $\dot{\eta}(t)$ and the horizontal space at $\eta(t).$\\
	Differentiating \eqref{||(F_*x)||^2}, we get
	$$\dfrac{d}{dt}g_Q(\pi_*X, \pi_*X) = 2ksin\omega(t)cos\omega(t)\dfrac{d\omega}{dt}$$
	which implies\\
	$$g_Q(\nabla_{\dot{\eta}}\pi_*X, \pi_*X) =  ksin\omega(t)cos\omega(t)\dfrac{d\omega}{dt}.$$
	Since
	\begin{equation}
		g_Q(\nabla_{\dot{\eta}}\pi_*X, \pi_*X) = g_Q(\nabla^Q_{\pi_*X}\pi_*X + \nabla^Q_U \pi_*X, \pi_*X).
	\end{equation}
	
	Using \eqref{COHM} and \eqref{DJF} in above equation, we have
	\begin{equation}\label{G10}
		\begin{split}
			g_Q(\nabla_{\dot{\eta}}\pi_*(X), \pi_*X)=& -g_Q(\nabla^Q_{\pi_*X}\alpha^2\pi_*X + \nabla^Q_{\pi_*X}\delta(\alpha \pi_*X) \\&+ \nabla^Q_{\pi_*X}B(\delta \pi_*X)+ \nabla^Q_{\pi_*X}C(\delta \pi_*X) + \nabla^Q_U \alpha^2\pi_*X \\&+ \nabla^Q_U \delta(\alpha \pi_*X) + \nabla^Q_U B(\delta \pi_*X) + \nabla^Q_U C(\delta \pi_*X), \pi_*X).
		\end{split}
	\end{equation}
	Making use of \eqref{Sv}, \eqref{SFF} and Theorem \ref{NSCfor Slant} in \eqref{G10}, we obtain
	\begin{equation}
		\begin{split}
			g_Q(\nabla_{\dot{\eta}}\pi_*(X), \pi_*X)= & g_Q(cos^2\theta \pi_*(\nabla^P_X X) + S_{\delta(\alpha \pi_*X)}\pi_*X \\&- \pi_*(\nabla^P_X *_{\pi_*}B(\delta \pi_*X)) + S_{C(\delta \pi_*X)}\pi_*X \\&+ cos^2\theta\nabla^Q_U \pi_*X - \nabla^Q_U B(\delta \pi_*X), \pi_*X) \\&= ksin\omega(t)cos\omega(t)\dfrac{d\omega}{dt}.
		\end{split}
	\end{equation}
	From \eqref{GCCS1}, we get
	
	\begin{equation}\label{theorem 2nd}
		\begin{split}
			&  g_Q(B(\nabla \pi_*)(X, *_{\pi_*}BU) + \alpha (\nabla^P_X *_{\pi_*}BU) - \alpha(S_{CU}\pi_*X) + B(\nabla^{\pi\perp}_X CU)\\& + \alpha (\nabla^Q_U BU) + B(\nabla^{\pi\perp}_U CU), \pi_*X) = ksin\omega(t)cos\omega(t)\dfrac{d\omega}{dt}.
		\end{split}
	\end{equation}

	Since $\pi$ is a Clairaut Riemannian map with $s=e^f$ if and only if $\dfrac{d}{dt}(e^{fo\eta}sin\omega(t)) = 0.$\\
	Therefore , 
	\begin{equation}
		e^{fo\eta}\dfrac{d}{dt}(fo\eta)sin\omega(t) + cos\omega(t)e^{fo\eta}\dfrac{d\omega}{dt} = 0,
	\end{equation}
	multiplying above equation with $ksin\omega(t)$, we get\\
	$$e^{fo\eta}(ksin^2\omega(t)\dfrac{d}{dt}fo\eta+ ksin\omega(t)cos\omega(t)\frac{d\omega}{dt}) = 0.$$
	Since $e^{fo\eta}$ is a positive function,
	$$g_Q(\pi_*{X}, \pi_*{Y})g_Q(gradf, \dot{\eta}) = -ksin\omega(t)cos\omega(t)\dfrac{d\omega}{dt}.$$
	From \ref{theorem 2nd}, we get
	\begin{equation}
		\begin{split}
			&g_Q(B(\nabla \pi_*)(X, *_{\pi_*}BU) + \alpha (\nabla^P_X *_{\pi_*}BU) - \alpha(S_{CU}\pi_*X) + B(\nabla^{\pi\perp}_X CU)\\& + \alpha(\nabla^Q_U BU) + B(\nabla^{\pi\perp}_U CU), \pi_*X) = - g_Q(\pi_*{X}, \pi_*{Y})g_Q(gradf, \dot{\eta}). 
		\end{split}
	\end{equation}
	This completes the proof.
	
\end{proof}
\begin{theorem}\label{condition on f}
	Let $\pi$ be a Clairaut slant Riemannian map from a Riemannian manifold $(P, g_P)$ to a K\"ahler manifold  $(Q, g_Q, J')$ with $s= e^f$ such that $\delta$ is parallel. Then $f$ is constant on $\delta(range\pi_*)$.
\end{theorem}
\begin{proof}
Since $\pi$ is Clairaut Riemannian map with $s=e^f,$ from \eqref{SFF}, Lemma \ref{Lemma} and Theorem \ref{NSC}, we get 
\begin{equation}\label{CRMC}
	\nabla^{Q_\pi}_X \pi_*Y - \pi_*(\nabla^P_X Y) = -g_P(X, Y)\nabla^Qf ~ for X, Y \in \Gamma(ker\pi_*)^\perp.
\end{equation}

Taking inner product of \eqref{CRMC} with $\delta\pi_*Z$, we get
\begin{equation*}
	g_Q(\nabla^{Q_\pi}_X \pi_*Y - \pi_*(\nabla_X Y), \delta\pi_*Z) = -g_P(X, Y)g_Q(\nabla^Q f, \delta \pi_*Z).		
\end{equation*}
Thus,
\begin{equation}\label{CRMI}
	g_Q(\nabla^{Q_\pi}_X \pi_*Y , \delta \pi_*Z) = -g_P(X, Y)g_Q(\nabla^Q f, \delta \pi_*Z).		
\end{equation}
Since $\nabla^Q$ is the Levi-Civita connection of $Q$ and $\nabla^{Q_\pi}$ is the pullback connection of $\nabla^Q,$ therefore  $\nabla^{Q_\pi}$ is also Levi-Civita connection on $Q$. Then using metric compatibility condition, we get
\begin{equation}
	-g_Q(\nabla^{Q_\pi}_X \delta \pi_*Z, \pi_*Y) = -g_P(X, Y)g_Q(\nabla^Q f, \delta \pi_*Z). 		
\end{equation}
Since $\delta$ is parallel,
\begin{equation}
	g_Q(\delta \nabla^{Q_\pi}_X \pi_*Z, \pi_*Y) = g_P(X, Y)g_Q(\nabla^Q f, \delta\pi_*Z).		
\end{equation}
It gives $$g_P(X, Y)g_Q(\nabla^Q f, \delta \pi_*Z) = 0,$$
which implies that $\delta \pi_*Z(f) = 0.$ 
This completes the proof.

\end{proof}

\begin{theorem}
	\label{THR3.4}
Let $\pi : (P, g_P) \rightarrow (Q, g_Q, J')$ be a Clairaut slant
Riemannian map with $s = e^f$ from a Riemannian manifold $P$ to a K\"ahler manifold $Q$. Then $\pi$ is totally geodesic if and only if the following conditions are satisfied:
\begin{itemize}
	\item [(i)] $ker\pi_*$ is totally geodesic,
	\item [(ii)] $(ker\pi_*)^\perp$ is totally geodesic,
	\item [(iii)] $cos^2\theta \pi_*(\nabla^P_X Y) - \nabla^{\pi\perp}_X \delta(\alpha \pi_*Y)- B\nabla^{\pi\perp}_X \delta \pi_*Y - C \nabla^{\pi\perp}_X \delta \pi_*Y - \pi_*(\nabla^P_X Y)  = 0.$
	
\end{itemize}
\end{theorem}
\begin{proof}
We know that $\pi$ is totally geodesic if and only if
\begin{equation}\label{tgv}
	(\nabla \pi_*)(U, V) = 0, 
\end{equation}
\begin{equation}\label{tghv}
	(\nabla \pi_*)(X, U) = 0,
\end{equation}
\begin{equation}\label{tgh}
	(\nabla \pi_*)(X, Y) = 0,
\end{equation}
for $U, V\in\Gamma(ker\pi_*)$ and $X, Y\in \Gamma(ker\pi_*)^\perp.$\\
From \eqref{tgv} and \eqref{tghv}, we get that fibers are totally geodesic,   $(ker\pi_*)^\perp$ is totally geodesic, respectively. From \eqref{tgh}, we have
\begin{equation}
	(\nabla \pi_*)(X, Y) = \nabla^{Q_\pi}_X \pi_*Y - \pi_*(\nabla^P_X Y). 
\end{equation} 
Using \eqref{COKM} and \eqref{DJF}, we have
\begin{equation*}
	\begin{split}
		(\nabla \pi_*)(X, Y)& = -J'\nabla^{Q_\pi}_X(\alpha \pi_*Y) - J'\nabla^{Q_\pi}_X(\delta \pi_*Y) - \pi_*(\nabla^P_X Y).		
	\end{split}		
\end{equation*}
Making use of \eqref{COKM} and Theorem (\ref{NSCfor Slant}), we get
\begin{equation*}
	(\nabla \pi_*)(X, Y) = cos^2\theta\nabla^{Q_\pi}_X \pi_*Y - \nabla^{Q_\pi}_X \delta(\alpha \pi_*Y) - J'\nabla^{Q_\pi}_{X}\delta \pi_*Y - \pi_*(\nabla^P_X Y).		
\end{equation*}
Applying  \eqref{S_V} and \eqref{SFF} in above equation, we have  
\begin{equation}
	\begin{split}
		(1-cos^2\theta)(\nabla \pi_*)(X, Y)& = cos^2\theta \pi_*(\nabla^P_X Y) + S_{\delta(\alpha \pi_*Y)}\pi_*X - \nabla^{\pi\perp}_X \delta(\alpha \pi_*Y) \\&+ J'(S_{\delta \pi_*Y}\pi_*X) - J'\nabla^{\pi\perp}_X \delta \pi_*Y - \pi_*(\nabla^P_X Y).
	\end{split} 
\end{equation}
Using \eqref{tgh}, we get 
\begin{equation}
	\begin{split}
		&cos^2\theta \pi_*(\nabla^P_X Y) + S_{\delta(\alpha \pi_*Y)}\pi_*X - \nabla^{\pi\perp}_X \delta(\alpha \pi_*Y) + J'(S_{\delta \pi_*Y}\pi_*X) - J'\nabla^{\pi\perp}_X \delta \pi_*Y\\& - \pi_*(\nabla^P_X Y) = 0.
	\end{split}
\end{equation}

From above equation and Theorem \ref{NSC}, we get \\
\begin{equation}\label{LLT}
	\begin{split}
		&cos^2\theta \pi_*(\nabla^P_X Y) - {\delta(\alpha \pi_*Y)}(f) \pi_*X - \nabla^{\pi\perp}_X \delta(\alpha \pi_*Y) \\&- J'\delta \pi_*Y(f)\pi_*X - B\nabla^{\pi\perp}_X \delta \pi_*Y - C \nabla^{\pi\perp}_X \delta \pi_*Y - \pi_*(\nabla^P_X Y) = 0.
	\end{split}		
\end{equation}		
From Theorem \eqref{condition on f} and \eqref{LLT}, we get the required result.
\end{proof}
\begin{theorem}

Let $\pi$ be a Clairaut slant Riemannian map from a Riemannian manifold $(P, g_P)$ to a K\"ahler manifold $(Q, g_Q, J')$ with $s=e^f$ such that $\delta$ is parallel. Then $Q$ is a locally product manifold of $Q_{(range\pi_*)} \times Q_{(range\pi_*)^\perp}$ if and only if 
\begin{equation}
	g_Q(\nabla^{\pi\perp}_{\pi_*X} \delta(\alpha \pi_*Y) + B\nabla^{\pi\perp}_{\pi_*X} \delta \pi_*Y + C\nabla^{\pi\perp}_{\pi_*X}\delta \pi_*Y, U) = 0
\end{equation}
for $\pi_*X, \pi_*Y \in\Gamma(range\pi_*)~ and ~ U\in\Gamma(range \pi_*)^\perp.$

\end{theorem}
\begin{proof}
Since $Q$ is K\"ahler manifold, we have
\begin{equation}\label{TGC}
	g_Q(\nabla^Q_{\pi_*X} \pi_*Y, U) = g_Q(\nabla^Q_{\pi_*X} J'\pi_*Y, J' U)
\end{equation}

for $\pi_*X, \pi_*Y \in \Gamma(range\pi_*)$ and $U \in\Gamma(range\pi_*)^\perp.$\\
Using \eqref{DJF} in \eqref{TGC}, we get
\begin{equation}
	g_Q(\nabla^Q_{\pi_*X} \pi_*Y, U) = -g_Q(\nabla^Q_{\pi_*X}\alpha^2\pi_*Y + \delta(\alpha \pi_* Y) + J'\nabla^Q_{\pi_*X}\delta \pi_*Y, U).
\end{equation}
Using Theorem \ref{NSCfor Slant} and \eqref{Sv}, we have
\begin{equation}
	\begin{split}
		g_Q(\nabla^Q_{\pi_*X} \pi_*Y, U) =& cos^2\theta g_Q(\nabla^Q_{\pi_*X} \pi_*Y, U) - g_Q(-S_{\delta(\alpha \pi_*Y)}\pi_*X \\&+ \nabla^{\pi\perp}_{\pi_*X}\delta(\alpha \pi_*Y) +J'(-S_{\delta \pi_*Y} \pi_*X + \nabla^{\pi\perp}_X \delta\pi_*Y ), U).	 		
	\end{split} 
\end{equation}
Using Theorem \ref{NSC}, we obtain
\begin{equation}
	\begin{split}
		sin^2\theta g_Q(\nabla^Q_{\pi_*X} \pi_*Y, U) =& -g_Q(\delta(\alpha \pi_*Y)(f)\pi_*X + \nabla^{\pi\perp}_{\pi_*X}\delta(\alpha \pi_*Y) \\&+J'(\delta \pi_* Y(f)\pi_*X + \nabla^{\pi\perp}_X \delta \pi_*Y), U).
	\end{split} 
\end{equation}
From Theorem \eqref{condition on f}, we get the required result.
\end{proof}
\begin{theorem}
Let $\pi$ be  a Clairaut slant Riemannian map from a Riemannian manifold $(P, g_P)$
to a K\"ahler manifold $(Q, g_Q, J')$ with $s=e^f.$ Then, $range \pi_*$ is integrable if and only if
\begin{equation}\label{IC}
	\begin{split}
		g_Q(\nabla^{\pi\perp}_Y\delta(\alpha \pi_*X) - g_P(Y, *_{\pi_*}B(\delta \pi_*X))\nabla^Qf+ \nabla^{\pi\perp}_Y C(\delta \pi_*X), U) \\= g_Q(\nabla^{\pi\perp}_X\delta(\alpha \pi_*Y) +\nabla^{\pi\perp}_X C(\delta \pi_*Y)- g_P(X, *_{\pi_*}B(\delta \pi_*Y))\nabla^Qf , U),
	\end{split}
\end{equation}
where $\pi_*X, \pi_*Y \in \Gamma(range \pi_*)$ and $U\in\Gamma(range\pi_*)^\perp.$
\end{theorem}
\begin{proof}
For $\pi_*X, \pi_*Y\in\Gamma(range \pi_*)$ and $U\in\Gamma(range\pi_*)^\perp$, we have 
\begin{equation}
	g_Q([\pi_*X, \pi_*Y], U) = g_Q(\nabla^Q_{\pi_*X}\pi_*Y, U) - g_Q(\nabla^Q_{\pi_*Y}\pi_*X, U).
\end{equation}
Making use of \eqref{COHM}, \eqref{DJF} and \eqref{DJV}, we have

\begin{equation}
	\begin{split}
		g_Q([\pi_*X, \pi_*Y], U) =& -g_Q(\nabla^Q_{\pi_*X}\alpha^2 \pi_*Y + \nabla^Q_{\pi_*X}\delta(\alpha \pi_*Y) + \nabla^Q_{\pi_*X}B(\delta \pi_*Y)\\& +\nabla^Q_{\pi_*X}C(\delta \pi_*Y), U)+ g_Q(\nabla^Q_{\pi_*Y}\alpha^2 \pi_*X \\&+ \nabla^Q_{\pi_*Y}\delta(\alpha \pi_*X)  + \nabla^Q_{\pi_*Y}B(\delta 
		\pi_*X)+\nabla^Q_{\pi_*Y}C(\delta \pi_*X), U).
	\end{split}
\end{equation}
Using Theorem \ref{NSCfor Slant}, we have 
\begin{equation}
	\begin{split}
		g_Q([\pi_*X, \pi_*Y], U) =& -g_Q(-cos^2\theta\nabla^Q_{\pi_*X}\pi_*Y + \nabla^Q_{\pi_*X}\delta(\alpha \pi_*Y)\\&+ \nabla^Q_{\pi_*X}B(\delta \pi_*Y)+\nabla^Q_{\pi_*X}C(\delta \pi_*Y), U) \\&+g_Q(-cos^2\theta\nabla^Q_{\pi_*Y}\pi_*X + \nabla^Q_{\pi_*Y}\delta(\alpha \pi_*X) \\&+ \nabla^Q_{\pi_*Y}B(\delta \pi_*X)+\nabla^Q_{\pi_*Y}C(\delta \pi_*X), U).
	\end{split}
\end{equation}
Using \eqref{S_V} and \eqref{SFF}, we get
\begin{equation}
	\begin{split}
		(1-cos^2\theta)g_Q(\nabla^Q_{\pi_*X}\pi_*Y - \nabla^Q_{\pi_*Y}\pi_*X, U) =& -g_Q(\nabla^{\pi\perp}_X\delta(\alpha \pi_*Y) \\&+ (\nabla \pi_*)(X, *_{\pi_*}B(\delta \pi_*Y))\\&+ \nabla^{\pi\perp}_X C(\delta \pi_*Y), U) \\&+ g_Q(\nabla^{\pi\perp}_Y\delta(\alpha \pi_*X) \\&+ (\nabla \pi_*)(Y, *_{\pi_*}B(\delta \pi_*X)) \\&+ \nabla^{\pi\perp}_Y C(\delta \pi_*X), U).		
	\end{split}
\end{equation}
Since $\pi$ is Clairaut Riemannian map, by using Theorem \ref{NSC}, we get \eqref{IC}. \\
This completes the proof.

\end{proof}

\begin{theorem}
\label{THRminimal}\cite{S25} Let $\pi $ be a Clairaut Riemannian map with $%
s=e^{f}$ between Riemannian manifolds $(P,g_{P})$ and $(Q,g_{Q},)$
such that $ker\pi _{\ast }$ is minimal. Then, $\pi $ is harmonic if and only
if $f$ is constant.
\end{theorem}

\begin{theorem}
Let $\pi $ be Clairaut slant Riemannian map with $s=e^{f}$ from a Riemannian
manifold $(P,g_{P})$ to a K\"ahler manifold $(Q,g_{Q},J^{\prime })$
such that $ker\pi _{\ast }$ is minimal. Then, $\pi $ is harmonic if and only
if 
\begin{equation*}
	\frac{\csc ^{2}\theta }{q}trace\left\{ -\nabla _{X}^{\pi \perp }\delta
	\alpha \pi _{\ast }X+\delta S_{\delta \pi _{\ast }X}\pi _{\ast }X-C\nabla
	_{X}^{\pi \perp }\delta \pi _{\ast }X\right\} =0,
\end{equation*}

where $X\in \Gamma (ker\pi _{\ast })^{\perp }.$
\end{theorem}

\begin{proof}
Using \ref{S_V} and the property $J^{\prime 2}=-I,$ we get%
\begin{equation*}
	(\nabla \pi _{\ast })(X,Y)=-J^{\prime }\nabla _{X}^{Q_{\pi }}J^{\prime }\pi
	_{\ast }Y-\pi _{\ast }(\nabla _{X}^{P}Y),
\end{equation*}

for any $X,Y\in \Gamma (ker\pi _{\ast })^{\perp }.$ Using \eqref{DJF}, \eqref%
{DJV}, \eqref{S_V} and Theorem \ref{NSCfor Slant}, we have%
\begin{align*}
(\nabla \pi _{\ast })(&X,Y) =\cos ^{2}\theta ((\nabla \pi _{\ast
})(X,Y)+\pi _{\ast }(\nabla _{X}^{P}Y))-\nabla _{X}^{Q_{\pi }}\delta \alpha
\pi _{\ast }Y \\
&+\alpha S_{\delta \pi _{\ast }Y}\pi _{\ast }X+\delta S_{\delta \pi _{\ast
	}Y}\pi _{\ast }X-B\nabla _{X}^{\pi \perp }\delta \pi _{\ast }Y-C\nabla
_{X}^{\pi \perp }\delta \pi _{\ast }Y-\pi _{\ast }(\nabla _{X}^{P}Y)
\end{align*}

which can be written as%
\begin{align}
\sin ^{2}\theta (\nabla \pi _{\ast })(&X,Y) =-\sin ^{2}\theta \pi _{\ast
}(\nabla _{X}^{P}Y)+S_{\delta \alpha \pi _{\ast }Y}\pi _{\ast }X-\nabla
_{X}^{\pi \perp }\delta \alpha \pi _{\ast }Y  \label{3.38} \notag\\
&+\alpha S_{\delta \pi _{\ast }Y}\pi _{\ast }X+\delta S_{\delta \pi _{\ast
	}Y}\pi _{\ast }X-B\nabla _{X}^{\pi \perp }\delta \pi _{\ast }Y-C\nabla
_{X}^{\pi \perp }\delta \pi _{\ast }Y.  
\end{align}

Taking $X$ instead of $Y$ in \eqref{3.38}, we have%
\begin{align*}
		\sin ^{2}\theta (\nabla \pi _{\ast })(&X,X) =-\sin ^{2}\theta \pi _{\ast
	}(\nabla _{X}^{P}X)+S_{\delta \alpha \pi _{\ast }X}\pi _{\ast }X-\nabla
	_{X}^{\pi \perp }\delta \alpha \pi _{\ast }X \\
	&+\alpha S_{\delta \pi _{\ast }X}\pi _{\ast }X+\delta S_{\delta \pi _{\ast
		}X}\pi _{\ast }X-B\nabla _{X}^{\pi \perp }\delta \pi _{\ast }X-C\nabla
	_{X}^{\pi \perp }\delta \pi _{\ast }X.
\end{align*}

Taking $range\pi _{\ast }$ and $(range\pi _{\ast })^{\perp }$ components, we
have%
\begin{equation}
	(\nabla \pi _{\ast })(X,X)^{range\pi _{\ast }}=-\pi _{\ast }(\nabla
	_{X}^{P}X)+\csc ^{2}\theta \left\{ S_{\delta \alpha \pi _{\ast }X}\pi _{\ast
	}X+\alpha S_{\delta \pi _{\ast }X}\pi _{\ast }X-B\nabla _{X}^{\pi \perp
	}\delta \pi _{\ast }X\right\}  \label{3.39}
\end{equation}

and 
\begin{equation*}
	(\nabla \pi _{\ast })(X,X)^{(range\pi _{\ast })^{\perp }}=\csc ^{2}\theta
	\left\{ -\nabla _{X}^{\pi \perp }\delta \alpha \pi _{\ast }X+\delta
	S_{\delta \pi _{\ast }X}\pi _{\ast }X-C\nabla _{X}^{\pi \perp }\delta \pi
	_{\ast }X\right\} .
\end{equation*}

We know that $range\pi _{\ast }$ part of $(\nabla \pi _{\ast })(X,X)$ is
equal to $0,$ i.e.%
\begin{equation*}
	\csc ^{2}\theta (S_{\delta \alpha \pi _{\ast }X}\pi _{\ast }X+\alpha
	S_{\delta \pi _{\ast }X}\pi _{\ast }X-B\nabla _{X}^{\pi \perp }\delta \pi
	_{\ast }X)-\pi _{\ast }(\nabla _{X}^{P}X)=0.
\end{equation*}
Therefore,%
\begin{equation}
	(\nabla \pi _{\ast })(X,X)=\csc ^{2}\theta \left\{ -\nabla _{X}^{\pi \perp
	}\delta \alpha \pi _{\ast }X+\delta S_{\delta \pi _{\ast }X}\pi _{\ast
	}X-C\nabla _{X}^{\pi \perp }\delta \pi _{\ast }X\right\}.   \label{3.40}
\end{equation}
Taking trace \eqref{3.40}, we have 
\begin{equation*}
	\nabla ^{Q}f=\frac{\csc ^{2}\theta }{q}trace\left\{ -\nabla _{X}^{\pi \perp
	}\delta \alpha \pi _{\ast }X+\delta S_{\delta \pi _{\ast }X}\pi _{\ast
	}X-C\nabla _{X}^{\pi \perp }\delta \pi _{\ast }X\right\} .
\end{equation*}

Since $\pi $ is a Clairaut slant Riemannian map with minimal fibers then
proof follows by Theorem \ref{THRminimal}.
\end{proof}

Now, we obtain two inequalities in terms of $(\nabla \pi _{\ast })(X,Y)$ of
Clairaut slant Riemannian map and check the equality case.

\begin{theorem}
Let $\pi $ be a Clairaut slant Riemannian map with $s=e^{f}$ from a Riemannian
manifold $(P,g_{P})$ to a K\"ahler manifold $(Q,g_{Q},J^{\prime }).$
Then we have%
\begin{eqnarray*}
	\sin ^{4}\theta \left\Vert (\nabla \pi _{\ast })(X,Y)\right\Vert ^{2} &\geq
	&\left\Vert S_{\delta \alpha \pi _{\ast }Y}\pi _{\ast }X\right\Vert
	^{2}+\left\Vert \nabla _{X}^{\pi \perp }\delta \alpha \pi _{\ast
	}Y\right\Vert ^{2} \\
	&&+\left\Vert B\nabla _{X}^{\pi \perp }\delta \pi _{\ast }Y\right\Vert
	^{2}+\left\Vert C\nabla _{X}^{\pi \perp }\delta \pi _{\ast }Y\right\Vert ^{2}
	\\
	&&+2\left\{ 
	\begin{array}{c}
		-\sin ^{2}\theta g_{Q}(\pi _{\ast }(\nabla _{X}^{P}Y),S_{\delta \alpha \pi
			_{\ast }Y}\pi _{\ast }X)-\sin ^{2}\theta g_{Q}(\pi _{\ast }(\nabla
		_{X}^{P}Y),\alpha S_{\delta \pi _{\ast }Y}\pi _{\ast }X) \\ 
		+\sin ^{2}\theta g_{Q}(\pi _{\ast }(\nabla _{X}^{P}Y),B\nabla _{X}^{\pi
			\perp }\delta \pi _{\ast }Y)+g_{Q}(S_{\delta \alpha \pi _{\ast }Y}\pi _{\ast
		}X,\alpha S_{\delta \pi _{\ast }Y}\pi _{\ast }X) \\ 
		-g_{Q}(S_{\delta \alpha \pi _{\ast }Y}\pi _{\ast }X,B\nabla _{X}^{\pi \perp
		}\delta \pi _{\ast }Y)-g_{Q}(\nabla _{X}^{\pi \perp }\delta \alpha \pi
		_{\ast }Y,\delta S_{\delta \pi _{\ast }Y}\pi _{\ast }X) \\ 
		+g_{Q}(\nabla _{X}^{\pi \perp }\delta \alpha \pi _{\ast }Y,C\nabla _{X}^{\pi
			\perp }\delta \pi _{\ast }Y)-g_{Q}(\alpha S_{\delta \pi _{\ast }Y}\pi _{\ast
		}X,B\nabla _{X}^{\pi \perp }\delta \pi _{\ast }Y) \\ 
		-g_{Q}(\delta S_{\delta \pi _{\ast }Y}\pi _{\ast }X,C\nabla _{X}^{\pi \perp
		}\delta \pi _{\ast }Y)%
	\end{array}%
	\right\}
\end{eqnarray*}

the inequality is satisfied if and only if $\pi $ is a Clairaut slant
Riemannian map. In the equality case, it takes the following form:%
\begin{eqnarray*}
	\sin ^{4}\theta \left\Vert (\nabla \pi _{\ast })(X,Y)\right\Vert ^{2}
	&=&\left\Vert S_{\delta \alpha \pi _{\ast }Y}\pi _{\ast }X\right\Vert
	^{2}+\left\Vert \nabla _{X}^{\pi \perp }\delta \alpha \pi _{\ast
	}Y\right\Vert ^{2} \\
	&&+\left\Vert B\nabla _{X}^{\pi \perp }\delta \pi _{\ast }Y\right\Vert
	^{2}+\left\Vert C\nabla _{X}^{\pi \perp }\delta \pi _{\ast }Y\right\Vert ^{2}
	\\
	&&+2\left\{ 
	\begin{array}{c}
		-\sin ^{2}\theta g_{Q}(\pi _{\ast }(\nabla _{X}^{P}Y),S_{\delta \alpha \pi
			_{\ast }Y}\pi _{\ast }X)-\sin ^{2}\theta g_{Q}(\pi _{\ast }(\nabla
		_{X}^{P}Y),\alpha S_{\delta \pi _{\ast }Y}\pi _{\ast }X) \\ 
		+\sin ^{2}\theta g_{Q}(\pi _{\ast }(\nabla _{X}^{P}Y),B\nabla _{X}^{\pi
			\perp }\delta \pi _{\ast }Y)+g_{Q}(S_{\delta \alpha \pi _{\ast }Y}\pi _{\ast
		}X,\alpha S_{\delta \pi _{\ast }Y}\pi _{\ast }X) \\ 
		-g_{Q}(S_{\delta \alpha \pi _{\ast }Y}\pi _{\ast }X,B\nabla _{X}^{\pi \perp
		}\delta \pi _{\ast }Y)-g_{Q}(\nabla _{X}^{\pi \perp }\delta \alpha \pi
		_{\ast }Y,\delta S_{\delta \pi _{\ast }Y}\pi _{\ast }X) \\ 
		+g_{Q}(\nabla _{X}^{\pi \perp }\delta \alpha \pi _{\ast }Y,C\nabla _{X}^{\pi
			\perp }\delta \pi _{\ast }Y)-g_{Q}(\alpha S_{\delta \pi _{\ast }Y}\pi _{\ast
		}X,B\nabla _{X}^{\pi \perp }\delta \pi _{\ast }Y) \\ 
		-g_{Q}(\delta S_{\delta \pi _{\ast }Y}\pi _{\ast }X,C\nabla _{X}^{\pi \perp
		}\delta \pi _{\ast }Y)%
	\end{array}%
	\right\}
\end{eqnarray*}
\end{theorem}

\begin{proof}
By taking the product of \eqref{3.38} by itself, we obtain%
\begin{eqnarray}
	\sin ^{4}\theta \left\Vert (\nabla \pi _{\ast })(X,Y)\right\Vert ^{2}
	&=&\sin ^{4}\theta \left\Vert \pi _{\ast }(\nabla _{X}^{P}Y)\right\Vert
	^{2}+\left\Vert S_{\delta \alpha \pi _{\ast }Y}\pi _{\ast }X\right\Vert
	^{2}+\left\Vert \nabla _{X}^{\pi \perp }\delta \alpha \pi _{\ast
	}Y\right\Vert ^{2}  \label{3.41} \\
	&&+\left\Vert B\nabla _{X}^{\pi \perp }\delta \pi _{\ast }Y\right\Vert
	^{2}+\left\Vert C\nabla _{X}^{\pi \perp }\delta \pi _{\ast }Y\right\Vert ^{2}
	\notag \\
	&&+2\left\{ 
	\begin{array}{c}
		-\sin ^{2}\theta g_{Q}(\pi _{\ast }(\nabla _{X}^{P}Y),S_{\delta \alpha \pi
			_{\ast }Y}\pi _{\ast }X)-\sin ^{2}\theta g_{Q}(\pi _{\ast }(\nabla
		_{X}^{P}Y),\alpha S_{\delta \pi _{\ast }Y}\pi _{\ast }X) \\ 
		+\sin ^{2}\theta g_{Q}(\pi _{\ast }(\nabla _{X}^{P}Y),B\nabla _{X}^{\pi
			\perp }\delta \pi _{\ast }Y)+g_{Q}(S_{\delta \alpha \pi _{\ast }Y}\pi _{\ast
		}X,\alpha S_{\delta \pi _{\ast }Y}\pi _{\ast }X) \\ 
		-g_{Q}(S_{\delta \alpha \pi _{\ast }Y}\pi _{\ast }X,B\nabla _{X}^{\pi \perp
		}\delta \pi _{\ast }Y)-g_{Q}(\nabla _{X}^{\pi \perp }\delta \alpha \pi
		_{\ast }Y,\delta S_{\delta \pi _{\ast }Y}\pi _{\ast }X) \\ 
		+g_{Q}(\nabla _{X}^{\pi \perp }\delta \alpha \pi _{\ast }Y,C\nabla _{X}^{\pi
			\perp }\delta \pi _{\ast }Y)-g_{Q}(\alpha S_{\delta \pi _{\ast }Y}\pi _{\ast
		}X,B\nabla _{X}^{\pi \perp }\delta \pi _{\ast }Y) \\ 
		-g_{Q}(\delta S_{\delta \pi _{\ast }Y}\pi _{\ast }X,C\nabla _{X}^{\pi \perp
		}\delta \pi _{\ast }Y)%
	\end{array}%
	\right\}  \notag
\end{eqnarray}

for any $X,Y\in \Gamma (ker\pi _{\ast })^{\perp }.$ From \eqref{3.41}, we get%
\begin{eqnarray*}
	\sin ^{4}\theta \left\Vert (\nabla \pi _{\ast })(X,Y)\right\Vert ^{2} &\geq
	&\left\Vert S_{\delta \alpha \pi _{\ast }Y}\pi _{\ast }X\right\Vert
	^{2}+\left\Vert \nabla _{X}^{\pi \perp }\delta \alpha \pi _{\ast
	}Y\right\Vert ^{2}+\left\Vert B\nabla _{X}^{\pi \perp }\delta \pi _{\ast
	}Y\right\Vert ^{2}+\left\Vert C\nabla _{X}^{\pi \perp }\delta \pi _{\ast
	}Y\right\Vert ^{2} \\
	&&+2\left\{ 
	\begin{array}{c}
		-\sin ^{2}\theta g_{Q}(\pi _{\ast }(\nabla _{X}^{P}Y),S_{\delta \alpha \pi
			_{\ast }Y}\pi _{\ast }X)-\sin ^{2}\theta g_{Q}(\pi _{\ast }(\nabla
		_{X}^{P}Y),\alpha S_{\delta \pi _{\ast }Y}\pi _{\ast }X) \\ 
		+\sin ^{2}\theta g_{Q}(\pi _{\ast }(\nabla _{X}^{P}Y),B\nabla _{X}^{\pi
			\perp }\delta \pi _{\ast }Y)+g_{Q}(S_{\delta \alpha \pi _{\ast }Y}\pi _{\ast
		}X,\alpha S_{\delta \pi _{\ast }Y}\pi _{\ast }X) \\ 
		-g_{Q}(S_{\delta \alpha \pi _{\ast }Y}\pi _{\ast }X,B\nabla _{X}^{\pi \perp
		}\delta \pi _{\ast }Y)-g_{Q}(\nabla _{X}^{\pi \perp }\delta \alpha \pi
		_{\ast }Y,\delta S_{\delta \pi _{\ast }Y}\pi _{\ast }X) \\ 
		+g_{Q}(\nabla _{X}^{\pi \perp }\delta \alpha \pi _{\ast }Y,C\nabla _{X}^{\pi
			\perp }\delta \pi _{\ast }Y)-g_{Q}(\alpha S_{\delta \pi _{\ast }Y}\pi _{\ast
		}X,B\nabla _{X}^{\pi \perp }\delta \pi _{\ast }Y) \\ 
		-g_{Q}(\delta S_{\delta \pi _{\ast }Y}\pi _{\ast }X,C\nabla _{X}^{\pi \perp
		}\delta \pi _{\ast }Y)%
	\end{array}%
	\right\}
\end{eqnarray*}

The inequality is satisfied if and only if $\pi $ is Clairaut slant
Riemannian map. In the equality case, it takes the following form:%

\begin{eqnarray*}
	\sin ^{4}\theta \left\Vert (\nabla \pi _{\ast })(X,Y)\right\Vert ^{2}
	&=&\left\Vert S_{\delta \alpha \pi _{\ast }Y}\pi _{\ast }X\right\Vert
	^{2}+\left\Vert \nabla _{X}^{\pi \perp }\delta \alpha \pi _{\ast
	}Y\right\Vert ^{2}+\left\Vert B\nabla _{X}^{\pi \perp }\delta \pi _{\ast
	}Y\right\Vert ^{2}+\left\Vert C\nabla _{X}^{\pi \perp }\delta \pi _{\ast
	}Y\right\Vert ^{2} \\
	&&+2\left\{ 
	\begin{array}{c}
		-\sin ^{2}\theta g_{Q}(\pi _{\ast }(\nabla _{X}^{P}Y),S_{\delta \alpha \pi
			_{\ast }Y}\pi _{\ast }X)-\sin ^{2}\theta g_{Q}(\pi _{\ast }(\nabla
		_{X}^{P}Y),\alpha S_{\delta \pi _{\ast }Y}\pi _{\ast }X) \\ 
		+\sin ^{2}\theta g_{Q}(\pi _{\ast }(\nabla _{X}^{P}Y),B\nabla _{X}^{\pi
			\perp }\delta \pi _{\ast }Y)+g_{Q}(S_{\delta \alpha \pi _{\ast }Y}\pi _{\ast
		}X,\alpha S_{\delta \pi _{\ast }Y}\pi _{\ast }X) \\ 
		-g_{Q}(S_{\delta \alpha \pi _{\ast }Y}\pi _{\ast }X,B\nabla _{X}^{\pi \perp
		}\delta \pi _{\ast }Y)-g_{Q}(\nabla _{X}^{\pi \perp }\delta \alpha \pi
		_{\ast }Y,\delta S_{\delta \pi _{\ast }Y}\pi _{\ast }X) \\ 
		+g_{Q}(\nabla _{X}^{\pi \perp }\delta \alpha \pi _{\ast }Y,C\nabla _{X}^{\pi
			\perp }\delta \pi _{\ast }Y)-g_{Q}(\alpha S_{\delta \pi _{\ast }Y}\pi _{\ast
		}X,B\nabla _{X}^{\pi \perp }\delta \pi _{\ast }Y) \\ 
		-g_{Q}(\delta S_{\delta \pi _{\ast }Y}\pi _{\ast }X,C\nabla _{X}^{\pi \perp
		}\delta \pi _{\ast }Y)%
	\end{array}%
	\right\}
\end{eqnarray*}

which completes the proof.
\end{proof}

\begin{theorem}
Let $\pi $ be a Clairaut slant Riemannian map with $s=e^{f}$ from a Riemannian
manifold $(P,g_{P})$ to a K\"ahler manifold $(Q,g_{Q},J^{\prime })$
such that $\pi $ is totally geodesic. Then we have%
\begin{eqnarray}
	&&\sin ^{4}\theta \left\Vert \pi _{\ast }(\nabla _{X}^{P}Y)\right\Vert
	^{2}+\left\Vert S_{\delta \alpha \pi _{\ast }Y}\pi _{\ast }X\right\Vert
	^{2}+\left\Vert \nabla _{X}^{\pi \perp }\delta \alpha \pi _{\ast
	}Y\right\Vert ^{2}  \label{3.42} \\
	&&+\left\Vert B\nabla _{X}^{\pi \perp }\delta \pi _{\ast }Y\right\Vert
	^{2}+\left\Vert C\nabla _{X}^{\pi \perp }\delta \pi _{\ast }Y\right\Vert ^{2}
	\notag \\
	&\leq &2\left\{ 
	\begin{array}{c}
		-\sin ^{2}\theta g_{Q}(\pi _{\ast }(\nabla _{X}^{P}Y),S_{\delta \alpha \pi
			_{\ast }Y}\pi _{\ast }X)-\sin ^{2}\theta g_{Q}(\pi _{\ast }(\nabla
		_{X}^{P}Y),\alpha S_{\delta \pi _{\ast }Y}\pi _{\ast }X) \\ 
		+\sin ^{2}\theta g_{Q}(\pi _{\ast }(\nabla _{X}^{P}Y),B\nabla _{X}^{\pi
			\perp }\delta \pi _{\ast }Y)+g_{Q}(S_{\delta \alpha \pi _{\ast }Y}\pi _{\ast
		}X,\alpha S_{\delta \pi _{\ast }Y}\pi _{\ast }X) \\ 
		-g_{Q}(S_{\delta \alpha \pi _{\ast }Y}\pi _{\ast }X,B\nabla _{X}^{\pi \perp
		}\delta \pi _{\ast }Y)-g_{Q}(\nabla _{X}^{\pi \perp }\delta \alpha \pi
		_{\ast }Y,\delta S_{\delta \pi _{\ast }Y}\pi _{\ast }X) \\ 
		+g_{Q}(\nabla _{X}^{\pi \perp }\delta \alpha \pi _{\ast }Y,C\nabla _{X}^{\pi
			\perp }\delta \pi _{\ast }Y)-g_{Q}(\alpha S_{\delta \pi _{\ast }Y}\pi _{\ast
		}X,B\nabla _{X}^{\pi \perp }\delta \pi _{\ast }Y) \\ 
		-g_{Q}(\delta S_{\delta \pi _{\ast }Y}\pi _{\ast }X,C\nabla _{X}^{\pi \perp
		}\delta \pi _{\ast }Y)%
	\end{array}%
	\right\}  \notag
\end{eqnarray}

for all $X,Y\in \Gamma (ker\pi _{\ast })^{\perp }.$
\end{theorem}

\begin{proof}
Taking into account of Theorem \ref{THR3.4} in \eqref{3.41}, we obtain \eqref%
{3.42}. This completes the proof.
\end{proof}
\begin{example}
	Let $P = Q = \mathbb{R}^4$ be Euclidean spaces with Riemannian metrics defined as 
	$$g_P = dx_1^2 + dx_2^2 + dx_3^2 + dx_4^2$$ and $$g_Q = dy_1^2 + dy_2^2 + dy_3^2 + dy_4^2,$$ respectively.\\
	\vspace{0.1cm}
We take the complex structure $J'$ on $Q$ as $J'(a, b, c, d) = (-c, -d, a, b)$. Then a basis of $T_rP$ is $\big\{e_i = \frac{\partial}{\partial x_i}$\} for i = 1...4\\ and basis of $T_{\pi(r)}Q$ is $\big\{e_j' = \frac{\partial}{\partial y_j}\}$ for j = 1...4. \\
	Now, we define a map $\pi : (P, g_P) \rightarrow (Q, g_Q, J')$ by
	$$\pi(x_1, x_2, x_3, x_4) = (\frac{x_1+x_2}{\sqrt{3}}, \frac{x_1+x_2}{\sqrt{6}}, 0, x_4).$$
	Then, we have $$ker\pi_* = \{U_1 = e_1-e_2, U_2 = e_3\}$$ 
	and  $$(ker\pi_*)^\perp = \{X_1 = e_1+e_2, X_2 = e_4\}.$$
Since, we have $g_P(X_i, X_j) = g_Q(\pi_* (X_i) , \pi_* (X_j))$ for i, j = 1, 2, 3, 4.\\ Thus $\pi$ is a Riemannian map and it can be easily seen that $$\pi_*(X_1) = \dfrac{2}{\sqrt{3}}e_1' + \dfrac{2}{\sqrt{6}}e_2'$$ and $$ \pi_*(X_2) = e_4'.$$
	Therefore $$range\pi_* = span\big\{U_1' = \dfrac{2}{\sqrt{3}}e_1' + \dfrac{2}{\sqrt{6}}e_1' , U_2' = e_4'\}$$ and $$(range \pi_*)^\perp = span\big\{X_1'= \dfrac{2}{\sqrt{3}}e_1' - \dfrac{2}{\sqrt{6}}e_1', X_2' = e_3'\}.$$
	Moreover, it is easy to see that $\pi$ is a slant Riemannian map with slant angle $cos^{-1}(\frac{1}{\sqrt{3}}).$
	Now to show $\pi$ is Clairaut Riemannian map we will find smooth function $f$ on $Q$ satisfying\\
	 $$(\nabla\pi_*)(X,X) =	-g_P(X,X)\nabla^Q f$$ for all $X \in \Gamma (ker\pi_*)^\perp.$\\
	 By using \eqref{SFF}, we obtain
	 $(\nabla \pi_*)(X,X) = 0$ for $X = a_1X_1 + a_2X_2,$ where $a_1, a_2 \in \mathbb{R}.$ Also
	 $$\nabla^Qf  = \sum_{i,j = 1}^{4}g^{ij}_Q\dfrac{\partial f}{\partial y_i}\dfrac{\partial}{\partial y_j}$$ $\implies$ $\nabla^Qf = 0$  for a constant function $f$.\\ Then
	 it is easy to verify that $$(\nabla \pi_*)(X,X) = -g_P(X,X)\nabla^Qf$$ for any vector field $X \in \Gamma (ker\pi_*)^\perp$ with a constant function $f$.\\ Thus $\pi$ is Clairaut slant
	 Riemannian map from a Riemannian manifold $(P, g_P)$ to a K\"ahler manifold
	 $(Q, g_Q, J')$.
	
\end{example}

\bigskip

\section{Acknowledgment}

First author is grateful to the financial support provided by CSIR (Council
of science and industrial research) Delhi, India. File
no.[09/1051(12062)/2021-EMR-I]. The second author is thankful to the
Department of Science and Technology(DST) Government of India for providing
financial assistance in terms of FIST project(TPN-69301) vide the letter
with Ref No.:(SR/FST/MS-1/2021/104).\\ 

\textbf{Declarations}

\textbf{Author's Contributions}: All authors contributed equally in this paper. All author read and approved the final manuscript.

\textbf{Funding}: No funding

\textbf{Availability of Data and Materials}: Not applicable.

\textbf{Ethical Approval}: Not required

\textbf{Competing Interests}: Not applicable

\textbf{Conflict of Interest}: The authors declare that they have no competing interest as defined by Springer.

\noindent J. Yadav and G. Shanker\newline
Department of Mathematics and Statistics\newline
Central University of Punjab\newline
Bathinda, Punjab-151401, India.\newline
Email: sultaniya1402@gmail.com; gauree.shanker@cup.edu.in\newline\\
\noindent M. Polat\\
Department of Mathematics, Faculty of Science, Dicle University, Sur, Diyarbak-\\
21280, Turkey. E-mail: murat.polat@dicle.edu.tr

\end{document}